\documentclass{amsart}
\usepackage{amsmath,amssymb}
\usepackage{graphicx,tikz,mathtools}

\usepackage{color}
\definecolor{Caca}{RGB}{193,124,250}

\newtheorem{theorem}{Theorem}[section]
\newtheorem{proposition}[theorem]{Proposition}
\newtheorem{corollary}[theorem]{Corollary}

\newtheorem{lemma}[theorem]{Lemma}

\theoremstyle{definition}

\theoremstyle{remark}
\newtheorem{remark}[theorem]{Remark}

\newcommand{\de}{\delta}
\newcommand{\ep}{\varepsilon}

\newcommand{\ga}{\gamma}

\newcommand{\la}{\lambda}
\newcommand{\om}{\omega}
\newcommand{\si}{\sigma}
\newcommand{\te}{\theta}
\newcommand{\vp}{\varphi}

\newcommand{\De}{\Delta}

\newcommand{\Si}{\Sigma}
\newcommand{\Om}{\Omega}

\def\NN{\mathbb{N}}
\def\RR{\mathbb{R}}

\def\ZZ{\mathbb{Z}}

\def\bD{\mathbb{D}}

\newcommand{\bP}{\mathbb{P}}

\renewcommand\SS{\mathbb{S}}
\newcommand{\cA}{{\mathcal A}}
\newcommand{\cB}{{\mathcal B}}
\newcommand{\cC}{{\mathcal C}}

\newcommand{\cJ}{{\mathcal J}}

\newcommand{\cN}{{\mathcal N}}

\newcommand{\cV}{{\mathcal V}}

\newcommand{\pd}{\partial}
\newcommand\minus\backslash

\newcommand\lan\langle
\newcommand\ran\rangle

\DeclareMathOperator\Div{div}

\DeclareMathOperator{\tr}{Tr}

\renewcommand\leq\leqslant
\renewcommand\geq\geqslant
\newlength{\intwidth}

\newcommand\BOm{\overline{\Om}}

\numberwithin{equation}{section}

 \DeclareMathOperator\curl{curl}

\DeclareMathOperator\Diff{Diff}
\newcommand\Per{\mathrm{Per}_{\mathrm{hyp}}}
\newcommand\diff{\Diff^+_\mu(\BSi)}

\newcommand\lB{{\bar B}}
\newcommand\tcB{\cB(\BOm)}
\newcommand\tB{\widetilde B}

\newcommand\BSi\Si %{\overline{\Si}}
\newcommand\cU{\mathcal{U}}

\newcommand\cUni{\cU_{\mathrm{NI}}}
\newcommand\MS{\cU_{\mathrm{MS}}}

\newcommand\NA{\cU_{\mathrm{twist}}}
\newcommand\cUi{\cU_{\mathrm{I}}}

\begin{document}

\title[Obstructions to topological relaxation for generic magnetic fields]{Obstructions to topological relaxation\\ for generic magnetic fields}

 \author{Alberto Enciso}
 \address{Instituto de Ciencias Matem\'aticas, Consejo Superior de
   Investigaciones Cient\'\i ficas, 28049 Madrid, Spain}
 \email{aenciso@icmat.es}

 \author{Daniel Peralta-Salas}
 \address{Instituto de Ciencias Matem\'aticas, Consejo Superior de
   Investigaciones Cient\'\i ficas, 28049 Madrid, Spain}
 \email{dperalta@icmat.es}

\newcommand{\cambios}[1]{\textcolor{Caca}{#1}}

%%    General info
%\subjclass[2010]{35B38, 58J05, 58K45}
%\date{\today}
%
%\keywords{ }
%
\begin{abstract}
For any axisymmetric toroidal domain $\Om\subset \RR^3$ we prove that there is a locally generic set of divergence-free vector fields that are not topologically equivalent to any magnetohydrostatic (MHS) equilibrium in $\Om$. Each vector field in this set is Morse--Smale on the boundary, does not admit a nonconstant first integral, and exhibits fast growth of periodic orbits; in particular this set is residual in the Newhouse domain. The key dynamical idea behind this result is that a vector field with a dense set of nondegenerate periodic orbits cannot be topologically equivalent to a generic MHS equilibrium. On the analytic side, this geometric obstruction is implemented by means of a novel rigidity theorem for the relaxation of generic magnetic fields with a suitably complex orbit structure.
\end{abstract}

\maketitle

\section{Introduction}

The topological complexity of magnetohydrostatic (MHS) equilibria (which formally solve the same equations as  3D stationary Euler flows) has played a central role in the geometric study of the MHD and Euler equations since Arnold's groundbreaking work in the 1960s. In spite of the large literature on this subject~\cite{Ar66,Moffatt,AK,Annals,Brenier,BW,BFV}, little is known about the accessibility of these equilibrium configurations, that is, of their connection with long-time limits of time-dependent solutions of the equations. This is not surprising, as extracting precise information about the long-time evolution of these PDEs is notoriously hard even in two dimensions~\cite{Drivas,KiselevSverak,Sverak}.

Magnetic relaxation is a mechanism that aims to obtain MHS equilibria as long-time limits of a topology-preserving evolution equation, hopefully easier to analyze that the original ideal MHD equations. Although the analysis of a number of magnetic relaxation equations has attracted considerable attention by themselves, they are introduced as a mean to obtain MHS equilibria (or stationary Euler flows) with prescribed streamline topology.

From a physical viewpoint, ideal magnetic relaxation plays a key role in the study of plasmas in stellar atmospheres~\cite{Cande}. In an influential paper from the early seventies, Parker proposed~\cite{Parker,Parker2} in this context that for most magnetic field topologies, there does not exist a smooth MHS equilibrium with that topology. As a consequence of this, he posited that the field should typically relax towards a state with tangential magnetic field discontinuities, a process that he termed ``topological dissipation'' because its onset is governed by the field line topology. Finding a rigorous proof (and even a precise statement) of Parker's hypothesis about the non-existence of MHS equilibria for most field topologies (which has
been referred to as ``Parker's problem'' or ``Parker's Magnetostatic Theorem'' in the literature; see e.g.~\cite{Low,Cande,PoHo}) is an open problem, and constitutes an important motivation for the theorem we prove in this article.

Our objective in this paper is to obtain generic obstructions for a divergence-free vector field~$B_0$ to be \emph{topologically equivalent} to some MHS equilibrium. By ``topologically equivalent'', we mean that there exists a $C^\infty$ volume-preserving diffeomorphism $X$ of the corresponding domain such that $\lB:= X_*B_0$ is an MHS equilibrium. This equivalence relation is also called {\em topological relaxation}\/ in the physics literature, to emphasize the fact that it is a purely topological property related to magnetic relaxation but independent of the specific model of magnetic relaxation one considers.

\subsection{Models for magnetic relaxation}

A non-resistive incompressible plasma is described by the MHD~system
\begin{gather}
	\pd_t B= B\cdot\nabla v - v\cdot\nabla B\,,\label{E.magdyn}\\
	\pd_t v + v\cdot\nabla v=\nu\De v+ B\times\curl B+\nabla P\,,\notag\\
	\Div B=\Div v=0\,. \notag
\end{gather}
This model is still somewhat artificial in that one takes zero resistivity to ensure that the evolution preserves the topology of~$B$ (i.e., that of the magnetic lines) but nonzero viscosity~$\nu>0$. However, this turns out to be a good approximation in the study of stellar atmospheres~\cite{Cande}. At the level of wishful thinking~\cite{Arnold, Moffatt}, one can expect that the inclusion of viscosity might force~$v$ to tend to~0 as $t\to\infty$, for certain initial data. If one assumes that $B(t,x)$ tends to a time-independent limit $\lB (x)$ and sets $v=0$, one formally obtains that the limit magnetic field satisfies the MHS equations:
\[
\lB\times\curl \lB+\nabla P=0\,,\qquad \Div \lB=0\,.
\]

From a mathematical point of view, a more convenient type of magnetic relaxation equation consists of the magnetic dynamo equation~\eqref{E.magdyn} with a simpler prescription for the divergence-free field~$v$. For example, the model
\begin{equation}\label{E.gamodel}
\pd_t B= B\cdot\nabla v - v\cdot\nabla B\,,\qquad v:=(-\De)^{-\ga}\,\bP (B\cdot\nabla B)\,,
\end{equation}
where $\bP$ denotes the Leray projector,
has attracted much attention since the work of Moffatt in the 1980s~\cite{Moffatt,Moffatt2}. For example, when $\ga=0$, the problem admits in 2D a certain kind of global dissipative weak solutions~\cite{Brenier}, however, not even local existence of strong solutions is known. When $\ga$ is large enough, it has been recently shown~\cite{BFV} that the equation is globally well-posed on Sobolev spaces, and that the field~$v$ actually tends to~0 as $t\to\infty$; this does not suffice to conclude that $B$ tends to an MHS equilibrium along a sequence of times tending to infinity.

For the purposes of this article, it is worth emphasizing that the special role played by the dynamo equation is due to the fact that it ensures that the topology of~$B$ is preserved by the evolution as long as the solutions remain smooth. More precisely, let $X^t$ denote the time $t$ flow of the vector field~$v$.
Then Equation~\eqref{E.magdyn} simply means that $B(x,t)$ can be written in terms of the initial datum $B_0(x)$ via the push-forward of the volume-preserving diffeomorphism $X^t$:
\[
B(t,x)=X^t_* B_0(x)\,,
\]
so that $B(t,\cdot)$ and $B_0$ are topologically equivalent. Of course, this does not ensure that the limit $\lim_{t\to \infty} B(t,\cdot)$, if it exists, must be topologically equivalent to~$B_0$ as well; although there are scant rigorous results in this direction, tangential field discontinuities and lost of field line topology are to be typically expected~\cite{Moffatt21,Komen22}. These discontinuities are a key ingredient in the mechanism Parker proposed to explain the heating of the solar corona~\cite{Parker,Parker2}.

\subsection{Obstructions to topological relaxation}

Let us now introduce the setting in which we shall consider the problem of topological relaxation. Throughout, $\Om\subset\RR^3$ will be an {\em axisymmetric toroidal domain with analytic boundary}. Thus~$\BOm$ is the body of revolution generated by the closed section
\[
\Si:=\BOm\cap\{x_1>0,\; x_2=0\}=\BOm\cap\{\te=0\}\,,
\]
which is diffeomorphic to the closed unit disk $\overline{\mathbb D}$ and does not intersect the $x_3$-axis.

In the solar physics literature it is customary to assume that the vector fields are {\em braided}, see e.g.~\cite{YH,Cande,Prior}. For the purposes of this paper, and without any essential loss of generality~\footnote{A generalized notion of ``braided'' vector field~\cite{Prior} assumes the existence of a disk-like foliation of $\Om$, transverse to the boundary, which is transverse to the field. Such a foliation is diffeomorphic to the standard foliation $\{\te=const\}$ we consider in this article, and all the proofs can be adapted to this setting with minor modifications.}, one defines the class of braided fields as the space of divergence-free vector fields with positive swirl function in $\Om$, that is,
\[
\tcB:=\{B_0\in C^\infty(\BOm,\RR^3): \Div B_0=0,\; B_0\cdot N=0,\; B_0^\te >0\}\,.
\]
Here $B_0^\te$ denotes the $\te$-component of the vector field~$B_0$ in cylindrical coordinates,
\[
B_0= B_0^r(r,\te,z)\, \partial_r+ B_0^\te(r,\te,z)\, \partial_\te + B_0^z(r,\te,z)\, \partial_z\,,
\]
and $N$ is the unit outer normal  vector on $\partial\Om$. Here $\{\partial_r,\partial_\theta,\partial_z\}$ is the usual holonomic basis in cylindrical coordinates. Note that we are {\em not}\/ assuming that $B_0$ is axisymmetric.

The following is the main result of this article, which characterizes a locally generic set of braided fields that are not topologically equivalent to any smooth MHS equilibrium. The characterization is done in terms of the Poincar\'e  (or first return) map of the braided field, which is well defined because of the transversality condition $B_0^\te>0$. For the statement, we recall that a diffeomorphism of the curve~$\pd\Si$, which is diffeomorphic to the circle~$\SS^1$, is {\em Morse--Smale}\/ if it has finitely many periodic points, all of which are hyperbolic (and therefore either expanding or contracting).

\begin{theorem}\label{T.main}
	The subset of braided fields that are not topologically equivalent to any smooth magnetohydrostatic equilibrium on~$\Om$ is locally generic in~$\tcB$.
	
	More precisely, let $B_0\in \tcB$ be a vector field whose Poincar\'e map $\Pi_{B_0}:\BSi\to\BSi$ satisfies the following conditions:
	\begin{enumerate}
		\item The restriction $\Pi_{B_0}|_{\pd\Si}$ is a Morse--Smale diffeomorphism of~$\pd\Si$.
		\item $\Pi_{B_0}$ is non-integrable. That is, there does not exist a nonconstant function $h\in C^\infty(\BSi)$ such that $h\circ \Pi_{B_0}=h$.
		\item For all $j,k\geq1$,
		\[
		\limsup_{n\to\infty} \frac{\#\Per(\Pi_{B_0},n)}{N_{j,k}(n)}=\infty\,,
		\]
		where $\#\Per(\Pi_{B_0},n)$ is the number of hyperbolic periodic points of the Poincar\'e map $\Pi_{B_0}$ whose least period is~$n$ and where $\{N_{j,k}(n)\}_{j,k,n=1}^\infty$ is certain fixed countable set of positive integers (whose dynamical meaning will become clear in the proof).
	\end{enumerate}
	Then there does not exist a smooth volume-preserving diffeomorphism $X:\BOm\to\BOm$ such that $\lB:= X_*B_0$ is an MHS equilibrium on~$\Om$, that is,
	\begin{equation}\label{E.statEuler}
	\lB\times \curl \lB+\nabla P=0\,,\qquad \Div \lB=0\,,\qquad \lB|_{\pd\Om}\cdot N=0		
	\end{equation}
	for some $P\in C^\infty(\BOm)$. Furthermore, the set of vector fields satisfying conditions (i)--(iii) is dense in a nonempty open subset of~$\tcB$.
\end{theorem}

\begin{remark}
In fact, Condition~(ii) can be substituted by the stronger but more transparent property that the periodic points of the diffeomorphism $\Pi_{B_0}$ are all nondegenerate (i.e., hyperbolic or elliptic) and dense in $\Sigma$. This easily implies that $\Pi_{B_0}$ is non-integrable, and is the main geometric obstruction to relaxation that underlies the theorem. One should emphasize that ergodicity cannot be used in this context to reach an analogous conclusion, since generic divergence-free fields are not ergodic as a consequence of Moser's twist theorem. Details are given in Section~\ref{S.non-integrable}.
\end{remark}
\begin{remark}
The open set of divergence-free vector fields where the fields that are not topologically conjugate to an MHS equilibrium is dense is directly connected with the so-called Newhouse domain of area-preserving diffeomorphisms of the disk~\cite{Duarte,Newhouse}, which we shall describe briefly in Appendix~\ref{A.Newhouse}. If the Newhouse domain turns out to be dense, as conjectured by some authors~\cite{Tura15}, one would infer that the set of vector fields satisfying conditions~(i)--(iii) is in fact dense in the whole space~$\tcB$.
\end{remark}
\begin{remark}
Since the space $\tcB$ of braided vector fields is an open subset of the space of smooth divergence-free vector fields on $\BOm$ that are tangent to the boundary, it is clear that the subset of fields that are not topologically equivalent to any smooth MHS equilibrium obtained in Theorem~\ref{T.main} is also locally generic in this space.
\end{remark}
\begin{remark}
In view of some applications in plasma physics, it would also be interesting to prove an analog of Theorem~\ref{T.main} for braided magnetic fields in axisymmetric toroidal shells, i.e., axisymmetric domains whose transverse section is diffeomorphic to an annulus instead of a disk. Most of the proof can be adapted to this setting, but an additional difficulty that arises is that the space of harmonic fields on $\Om$ has dimension 2 in this case, so the proof of Lemma~\ref{L.laeigen} breaks down in this setting.
\end{remark}

The interest of Theorem~\ref{T.main} is threefold. First, it provides a rigorous formulation of the fact that for ``most'' well-behaved initial data one cannot hope to obtain equilibrium states of the same topology as a long-time limit of the  solution to the corresponding Cauchy problem, and that this property is independent of the specific evolution equations one considers. Second, while there are many heuristic arguments about why topological magnetic relaxation should be generically impossible~\cite{Parker2,Moffatt21}, there are essentially no rigorous results in this direction. Third, as we shall see shortly, the argument we use to prove the statement is rather involved, and combines ideas about generic volume-preserving dynamical system, spectral theory, and some symmetry properties of MHS equilibria; we think that some of these methods may be useful to study different topological problems in plasma physics.

A final observation is that there are elementary obstructions that prevent a magnetic field from relaxing to a smooth MHS equilibrium of the same topology. Specifically, a lemma due to Cieliebak and Volkov~\cite[Lemma 2.3]{Cieliebak} ensures that if the magnetic field~$B_0$ exhibits a Reeb component, $B_0$ cannot be topologically equivalent to an MHS equilibrium, and one can use the theory of plugs to obtain significant generalizations of this~\cite{PRT}. For the benefit of the reader, in Appendix~\ref{A.Reeb} we recall the proof of~\cite[Lemma 2.3]{Cieliebak} and discuss some implications for relaxation using scaled Wilson-type plugs. Note, however, that these results have no bearing in the context considered in this paper because braided fields cannot exhibit any Reeb components.

This article is organized as follows. We present the proof of Theorem~\ref{T.main} in Section~\ref{S.outline}. To streamline the presentation, the proofs of many auxiliary results are relegated to Sections~\ref{S.cJ},~\ref{S.non-integrable} and~\ref{S.rigidity}. In Section~\ref{S.cJ} we show how to obtain genericity results for braided fields using diffeomorphisms of the disk. The fact that most braided fields are non-integrable is established in Section~\ref{S.non-integrable}. Lastly, in Section~\ref{S.rigidity} we prove a rigidity theorem for the topological relaxation of generic braided fields, which shows that under certain dynamical assumptions, the only MHS equilibria the field can relax to are a certain kind of Beltrami fields. The paper concludes with three appendices, where we recall the aforementioned lemma of Cieliebak and Volkov~\cite{Cieliebak}, record some facts about the so-called Newhouse domain of area-preserving diffeomorphisms of the disk and show that the obstructions for topological relaxation presented in Theorem~\ref{T.main} also work in the setting of ideal compressible MHD relaxation.

\section{Proof of the main theorem}
\label{S.outline}

In this section we prove Theorem~\ref{T.main} modulo several technical results that will be established in forthcoming sections. We naturally endow the space of braided fields $\tcB$ with the $C^\infty$ topology. Since the space of $C^\infty$ divergence-free vector fields on $\BOm$ tangent to the boundary is a complete Baire metric space, its open subset $\tcB$ with the induced topology is a Baire metric space too. Throughout, we say that a subset of $\tcB$ is {\em locally generic}\/ if it is residual in some open subset of $\tcB$. We recall that a set is {\em residual}\/ if it is a countable intersection of sets whose interior is dense. In particular, residual sets of~$\tcB$ are dense.

In the proof of Theorem~\ref{T.main}, a major role is played by the Poincar\'e map of a field~$B\in\tcB$. To define it,
consider the disk-like closed section
\[
\Si:=\BOm\cap\{\te=0\}
\]
defined in the Introduction,
which one can naturally parametrize by means of the two remaining cylindrical coordinates $(r,z)$. We recall that, by hypothesis, $\BSi$ is contained in the region $\{r>0\}$. Given any vector field $B\in \tcB $, let $\phi^t$ denote its time~$t$ flow. Since $B^\te>0$ on $\BOm$, for each point~$x$ in the surface~$\BSi$ there is a smallest positive time $\tau(x)$ for which $\phi^{\tau(x)}x$ is again in~$\BSi$. As usual, the map $\Pi_B: \BSi\to \BSi$ given by
\[
\Pi_B(x):= \phi^{\tau(x)}x
\]
is called the {\em Poincar\'e map}\/ (or first return map) of the field~$B$ on the closed section~$\BSi$.

It is well known that $\Pi_B$ is not just any kind of map from~$\BSi$ onto itself. To make this precise, let us introduce some notation. Given an area measure~$\mu$ on~$\BSi$, let us denote by $\Diff_\mu^+(\BSi)$ the space of $C^\infty$ diffeomorphisms of~$\BSi$ that preserve the measure~$\mu$ and are isotopic to the identity (as usual, we endow $\Diff_\mu^+(\BSi)$ with the $C^\infty$ topology, which makes it a complete Baire metric space). Likewise, for each $B\in\tcB $, we consider the measure on~$\Si$ given by
\[
\mu_B:=  {B^\te(r,0,z)}\,r\, dr\,dz\,.
\]
We then have the following easy result. Although this is well known to experts, for completeness we provide a short proof of this fact in Section~\ref{S.cJ}.

\begin{proposition}\label{P.Pi}
	For each $B\in\tcB $, $\Pi_B\in \Diff^+_{\mu_B}(\BSi)$.
\end{proposition}

\subsection{Addressing genericity questions via Poincar\'e maps}

We shall divide the proof of Theorem~\ref{T.main} in three steps. The first step, which we present in this subsection, is to show that one can effectively address genericity questions about braided fields by means of their Poincar\'e maps. Our objective in this subsection is to make this statement precise. The proofs of the various lemmas that we will need for this purpose are given in Section~\ref{S.cJ}.

A first difficulty one must address is that,
by Proposition~\ref{P.Pi}, the Poincar\'e maps of different vector field $B\in\tcB$ preserve, in general, different area measures on~$\BSi$. However, thanks to Moser's isotopy lemma~\cite{Moser}, one can consider a conjugation of the Poincar\'e map which preserves a fixed measure on~$\BSi$, for instance the natural area measure
\[
\mu:=r\, dr\, dz
\]
induced from the Euclidean space.
We will informally refer to the resulting diffeomorphism, which we shall denote by $\cJ(B)$, as the {\em $\mu$-preserving Poincar\'e map}\/ of~$B$. To ensure that these diffeomorphisms are close whenever the associated fields are close, we will  show that we can define the $\mu$-preserving Poincar\'e map of a field by means of a continuous function $\tcB\to\diff$. The proof of this lemma, which uses a result of Dacorogna and Moser~\cite{DM}, is presented in Section~\ref{S.cJ}.

\begin{lemma}\label{L.mapJ}
There exists a continuous map $\cJ: \tcB \to \diff$ such that, for each~$B\in\tcB$, there is a $C^\infty$ diffeomorphism $Y_B$~of~$\BSi$ conjugating the diffeomorphism~$\cJ(B)$ and the Poincar\'e map of~$B$:
	\begin{equation}\label{E.cJYB}
	\cJ(B)= Y_B\circ \Pi_B\circ Y_B^{-1}\,.
	\end{equation}
\end{lemma}

For our purposes, the key property of the map $\cJ$ that we shall also establish in Section~\ref{S.cJ} is the following. In its proof, a recent theorem of Treschev~\cite{Treschev} is crucial. Incidentally, we mention that Treschev's construction has also been key to the study of non-mixing and non-ergodicity properties of time-dependent 3D Euler flows presented in~\cite{KKPS}.

\begin{lemma}\label{L.residual}
The map $\cJ$ is onto and open. In particular, if $\cU$ is a residual subset of an open set $U\subset \diff$, then $\cJ^{-1}(\cU)$ is a residual subset of the open set~$\cJ^{-1}(U)\subset \tcB$.
\end{lemma}

\subsection{Generic braided fields can only relax to a Beltrami field}\label{SS.NI}

The second step of the proof is to show that, if the Poincar\'e map of a field $B\in\tcB$ is non-integrable, in a precise sense which we will show to hold true generically, then the only MHS equilibria that~$B$ can be topologically equivalent to are Beltrami fields\footnote{We recall that a vector field $B$ is Beltrami if it satisfies the equation $\curl B=\lambda B$ in $\Om$ for some constant $\la\neq 0$.}. Moreover, the proportionality factors of the Beltrami fields are not arbitrary, as they can only range over a certain countable set of eigenvalues. This should be understood as a rigidity result for MHS equilibria on axisymmetric toroidal domains that holds under the topological assumptions that the field is braided and its Poincar\'e map exhibits certain generic complex behavior\footnote{Without some hypotheses, the rigidity result is obviously not true, as not all MHS equilibria are Beltrami fields.}.

In the following theorem, which we will prove in Section~\ref{S.non-integrable}, we show that there is a residual set $\cUni$ of area-preserving diffeomorphisms of~$\BSi$ such that each map of $\cUni$ is non-integrable and Morse--Smale on the boundary. This will be a key ingredient in the proof of the generic rigidity result that lies at the heart of our main theorem.

\begin{theorem}\label{T.cU}
	Let~$\cUni $ be the set of maps $F\in \diff$ such that:
	\begin{enumerate}
		\item The restriction $F|_{\pd\Si}$ is a Morse--Smale diffeomorphism of~$\pd\Si$.
		\item $F$ is non-integrable. That is, there does not exist a nonconstant function $h\in C^\infty(\BSi)$ such that $h\circ F=h$.
	\end{enumerate}
Then $\cUni $ is residual in $\diff$.
\end{theorem}

\begin{remark}\label{R.cJPi}
The non-integrability condition is invariant under conjugation by a diffeomorphism of~$\BSi$. Indeed, by Equation~\eqref{E.cJYB}, $h\in C^\infty(\BSi)$ is invariant under~$\cJ(B)$  if and only if $h\circ Y_B\in C^\infty(\BSi)$ is invariant under $\Pi_B=Y_B^{-1}\circ \cJ(B)\circ Y_B$. Of course, the Morse-Smale property at the boundary is also invariant under conjugation.
\end{remark}

\begin{remark}
In the plasma physics literature, the property that is usually considered to conclude that a magnetic field is non-integrable is the ergodicity of the field~\cite{Pontin}, which implies the existence of a full measure set of magnetic lines that are dense on the domain. However, since this property is {\em not}\/ generic in the $C^\infty$ topology due to Moser's twist theorem, we cannot effectively use it in this context. Instead, the proof of Theorem~\ref{T.cU} is based on a novel different dynamical obstruction, which arises when the periodic orbits are all nondegenerate (i.e., hyperbolic or elliptic) and their union is dense in $\Om$. We shall see in Section~\ref{S.non-integrable} that this property is indeed generic in the smooth topology.
\end{remark}

To state the generic rigidity result, we need to introduce some notation. By Hodge theory, it is well known that the space of harmonic fields on~$\Om$ that are tangent to the boundary is one-dimensional (see e.g.~\cite{Acta}). We henceforth fix a nonzero harmonic field
\begin{equation}\label{E.defhOm}
h_\Om:=\frac{1}{r^2}\, \partial_\te\,,
\end{equation}
which satisfies
\begin{gather*}
\Div h_\Om=0\quad \text{and}\quad \curl h_\Om=0\qquad \text{in }\Om\,,\\
h_\Om\cdot N=0 	\quad \text{on }\pd \Om\,.
\end{gather*}

For any constant~$\la$, it is well known (see e.g.~\cite[Proposition 6.1]{Acta}) that the boundary value problem
\begin{gather}\label{E.eigen}
\curl v=\la v\quad \text{in }\Om\,,\qquad v\cdot N=0\quad \text{on } \pd \Om\,,\qquad \int_{\Om} v\cdot h_\Om\, dx=0
\end{gather}
admits nontrivial solutions if and only if $\la$ belongs to a countable subset of the real line without accumulation points, which we shall henceforth denote by $\{\la_j\}_{j=1}^\infty\subset\RR\backslash\{0\}$. For short, we will refer to these constants (which are both positive and negative) as the {\em eigenvalues}\/ of $\curl$ on~$\Om$.

In what follows, we denote by $d_j$ the {\em multiplicity}\/ of the eigenvalue~$\la_j$, that is, the dimension of the vector space of solutions to the problem~\eqref{E.eigen} with $\la=\la_j$ (which is necessarily finite). We will also fix a basis for the corresponding eigenspace, $\{v_{j,1},\dots, v_{j,d_j}\}$. Note that, while these are smooth divergence-free vector fields tangent to the boundary, they are not necessarily braided. Finally, we let~$V_j$ be the unique solution (see~\cite[Theorem~1]{Boul}) to the problem
\begin{gather*}
\curl V_j=\la_j V_j\quad \text{in }\Om\,,\qquad V_j\cdot N=0\quad \text{on } \pd \Om\,,\\
 \int_{\Om} V_j\cdot v_{j,k}\, dx=0\qquad \text{for all } k\in\{1,\cdots, d_j\}\,,\\
 \int_{\Om} V_j\cdot h_\Om\,dx=1\,.
\end{gather*}

The rigidity result we need, which we will prove in Section~\ref{S.rigidity}, can then be stated as follows:

\begin{theorem}\label{T.rigidity}
Assume that $B\in \cJ^{-1}(\cUni )$. If there exists a $C^\infty$ volume-preserving diffeomorphism~$Y$ of~$\BOm$ such that $\lB:=Y_*B$ is an MHS equilibrium (that is, $\lB$ satisfies Equation~\eqref{E.statEuler} for some pressure function $P\in C^\infty(\BOm)$), then
	\begin{equation}\label{E.formu}
			\lB= c_0V_j+\sum_{k=1}^{d_j}c_k \, v_{j,k}=:u^j_c
	\end{equation}
	for some integer $j\geq1$ and some real constants $c:= (c_0,c_1,\dots, c_{d_j})$.
\end{theorem}

\begin{remark}\label{R.relax}
This is a sufficient, and in fact generic, topological condition on the magnetic field (namely, being braided and having non-integrable Poincar\'e map that is Morse--Smale on the boundary) which ensures that the only MHS equilibria that the field can topologically relax to are Beltrami fields.
\end{remark}

\subsection{Fast growth for the number of periodic points and conclusion of the proof}

To close the argument, we need to use the rigidity result proven in Theorem~\ref{T.rigidity} to show that a (locally) generic braided field cannot be topologically equivalent to an MHS equilibrium. Intuitively speaking, this is easy to believe because we have shown that, generically, the only MHS equilibria that can arise are the Beltrami fields~$u^j_c$, which is a fairly small class of analytic vector fields. However, since there is no a priori control on the volume-preserving $C^\infty$~diffeomorphisms of the conjugation, proving that this is indeed the case is rather tricky. As we will see, our proof ultimately relies on the fact that typical area-preserving diffeomorphisms of~$\BSi$ exhibit an extremely fast growth of the number of (hyperbolic and elliptic) periodic points, in a sense that will eventually clash with the fact that all the complexity of the Poincar\'e map of the field must be ultimately encoded in the parameters~$j,c$ that describe the vector field~$u^j_c$.

To make precise this argument we start by noting that a Beltrami field~$u^j_c$ of the form~\eqref{E.formu} does not generally admit a global disk-like transverse section. However, if a field admits such a section, the same holds for any other field proportional to it, and furthermore both fields give rise to the same return map on the section. To factor in this, for each $j\geq1$ we consider the set of parameters $c=(c_0,c_1,\dots, c_{d_j})$, normalized to have unit length, for which the corresponding vector field $u^j_c$ exhibits an embedded (compact) global transverse section $\Si^j_c\cong \overline{\mathbb D}$ with $\partial\Si^j_c\subset\partial\Om$:
\[
\cC_j:=\{ c\in \SS^{d_j}: u^j_c \text{ admits a global transverse section } \Sigma^j_c\subset\overline\Om \text{ diffeomorphic to } \overline{\bD}\}\,.
\]
Here $\SS^{d_j}:=\{c\in\RR^{d_j+1}: |c|^2=1\}$. The set $\cC_j$ is obviously open and by the density of analytic functions there is no loss of generality in assuming that the section $\Sigma^j_c$ is an analytic submanifold.

Since any open subset of a manifold is $\sigma$-compact, we can fix a countable cover of~$\cC_j$ by compact sets, which we henceforth denote by $\{K_{j,k}\}_{k=1}^\infty$ with $K_{j,k}\subset\cC_j$. Without any loss of generality, we can assume that these compact sets are small enough so that that there is an analytic disk-like submanifold $\Si_{j,k}\cong \overline{\bD}$ that is a global transverse section of~$u^j_c$ for all $c\in K_{j,k}$. Let
$$
\Pi_{j,k,c}:\Si_{j,k}\to \Si_{j,k}
$$
be the first return map of the vector field $u^j_c$ with $c\in K_{j,k}$. This map is a diffeomorphism of~$\Si_{j,k}$ that preserves an area measure $\mu_{u^j_c}$ by the proof of Proposition~\ref{P.Pi}, the explicit expression of which will not be relevant for our purposes. A straightforward property that we will crucially use in the proof is that the number of hyperbolic periodic points with fixed period of $\Pi_{j,k,c}$ is uniformly bounded for all $c\in K_{j,k}$. More precisely, for $c\in K_{j,k} $ let us denote by
\[
\text{Per}(\Pi_{j,k,c},n):=\{x\in\Sigma^j_k: \Pi_{j,k,c}^n(x)=x \, \text{ and } \, \Pi_{j,k,c}^m(x)\neq x\, \text{ for } 0<m<n\}
\]
the set of periodic points of the diffeomorphism~$\Pi_{j,k,c}$ whose least period is~$n$, and by $\Per(\Pi_{j,k,c},n)\subset \text{Per}(\Pi_{j,k,c},n)$ the set of hyperbolic periodic points of minimal period $n$. Recall that $x$ is a hyperbolic $n$-periodic point of a diffeomorphism $F$ if none the eigenvalues of the derivative~$\nabla F^n(x)$ has unit modulus. Then we have the following uniform bound:

\begin{lemma}\label{L.N}
	For any $j,k$,
	\[
N_{j,k}(n):= \sup\left\{ \#\Per(\Pi_{j,k,c},n): c\in K_{j,k}\right\}<\infty\,.
\]
\end{lemma}

\begin{proof}
As the vector field~$u^j_c$ is a Beltrami field with eigenvalue $\la_j$ on the toroidal domain $\Om$, cf.~Equation~\eqref{E.formu}, and the boundary~$\partial\Om$ is analytic by assumption, it follows that $u^j_c$ is analytic up to the boundary by~\cite[Theorem A.1]{JEMS}. The flow defined by $u^j_c$ is then analytic, and since we chose the disk-like global transverse section $\Sigma_{j,k}$ to be a compact analytic  submanifold with boundary, it easily follows that the Poincar\'e map $\Pi_{j,k,c}$ is an analytic diffeomorphism. Since $\Per(\Pi_{j,k,c},n)$ is a subset of the set of isolated points of $\text{Per}(\Pi_{j,k,c},n)$, and the latter set is finite by the analyticity of $\Pi_{j,k,c}^n$ on $\Sigma^j_k$ for any fixed $n$, $j$, $k$ and~$c\in\cC_j$, cf.~\cite[Proposition~11]{Bruhat}, the finiteness of $N_{j,k}(n)$ then follows from the compactness of~$K_{j,k}$. The lemma is proven.
\end{proof}

The reason for which the uniform bounds for $N_{j,k}(n)$ are useful in our argument is that they are going to clash with the fact that (locally) generic area-preserving diffeomorphisms of the disk exhibit an arbitrarily fast growth of the number of periodic points, as made precise in the following theorem due to Asaoka:

\begin{theorem}[Asaoka~\cite{Asaoka}]\label{T.Asaoka}
For any $j,k$, there is a subset $\cU_{j,k}\subset \diff$ such that each $F\in \cU_{j,k}$ satisfies
	\[
	\limsup_{n\to\infty}\frac{\#\Per(F,n)}{N_{j,k}(n)}=\infty\,.
	\]
This subset is residual (and therefore dense) in an open subset $\cN_\Sigma$ of $\diff$ that does not depend on $j,k$. Moreover, a dense set of diffeomorphisms in $\cU_{j,k}$ can be taken to be analytic.
\end{theorem}
\begin{remark}
As observed by P. Berger~\cite{Berger}, we can take the open set $\cN_\Sigma$ to be the Newhouse domain of $\diff$. For the benefit of the reader, in Appendix~\ref{A.Newhouse} we recall a few basic properties of this open set of area-preserving diffeomorphisms.
\end{remark}

We already have all the information we need to prove Theorem~\ref{T.main}. Specifically, as residual sets in the space of braided fields~$\tcB$ are dense because of the Baire property, Theorem~\ref{T.main} now follows from the following fact:

\begin{proposition}\label{P.final}
Consider the set
\[
\cV:=\cJ^{-1}\left(\cUni \cap \bigcap_{j,k=1}^\infty \cU_{j,k}\right)
\]
of braided fields in $\overline{\Om}$. Then this set is residual in the open set $\cJ^{-1}(\cN_\Sigma)\subset\tcB$. Moreover, for any braided field $B\in\cV$, there does not exist a smooth diffeomorphism $Y$ of~$\overline\Omega$ such that $Y_*B$ is an MHS equilibrium.
\end{proposition}

\begin{proof}
We start by noticing that $\cJ^{-1}(\cN_\Sigma)$ is open because $\cN_\Sigma$ is an open subset of $\diff$ and $\cJ$ is a continuous map by Lemma~\ref{L.mapJ}. Moreover, since the countable intersection of residual sets is residual, then Theorems~\ref{T.cU} and~\ref{T.Asaoka} imply that
$$\cUni \cap \bigcap_{j,k=1}^\infty \cU_{j,k}$$
is residual in $\cN_\Sigma$, and therefore $\cV$ is residual in~$\cJ^{-1}(\cN_\Sigma)$ by Lemma~\ref{L.residual}.
	
Let $B\in\cV$, and assume that there is a smooth diffeomorphism $Y$ of~$\overline\Omega$ such that $\lB:=Y_*B$ is an MHS equilibrium. Since $\cJ(B)\in\cUni $, Theorem~\ref{T.rigidity} ensures that $\lB=u^j_c$ for some constants $j,c$. $B$ being braided, it is clear that $\lB$ admits the disk-like global transverse section $Y(\Sigma)$, and hence $c/|c|\in \mathcal C_j$. If $K_{j,k}$ is any element of the cover of~$\cC_j$ which contains the unit vector $c/|c|$, we then infer that
	\[
	\# \Per(\cJ(B),n)=\#\Per(\Pi_{j,k,c},n)\leq N_{j,k}(n)
	\]
	for all~$n$, where we have used that the number of hyperbolic periodic points of the Poincar\'e map of $B$ is invariant under diffeomorphisms of the vector field, and does not depend on the particular disk-like global section that is taken (here $\Sigma_{j,k}$ rather than $Y(\Sigma)$; see e.g.~\cite[Chapter~3.1]{Palis}). This contradicts the assumption that $\cJ(B)\in\cU_{j,k}$, which completes the proof of the proposition.
\end{proof}

\begin{remark}\label{R.Newhouse}
If the Newhouse domain turns out to be dense in the space of $C^\infty$ area-preserving diffeomorphisms of the disk, as conjectured by some authors~\cite{Tura15}, we obtain that the set of braided fields~$\cV$ that cannot topologically relax to an MHS equilibrium is dense in the whole space of braided fields~$\tcB$.
\end{remark}

\begin{remark}\label{R.analytic}
In addition to the hyperbolicity of the periodic points we are counting and the compactness of the set~$K_{j,k}$, the proof of Lemma~\ref{L.N} only uses one nontrivial ingredient: the fact that the Beltrami fields~$u^j_c$ are analytic up to the boundary. In the proof of Proposition~\ref{P.final}, the connection between the analytic fields~$u^j_c$ and a residual set of $C^\infty$~braided fields is achieved via our generic rigidity result, cf. Theorem~\ref{T.rigidity}. It is tempting to lean more heavily into techniques from the theory of analytic dynamical systems and try to obtain an analog of Theorem~\ref{T.main} without using any rigidity results, only abstract properties of analytic diffeomorphisms. However, the whole proof breaks down because the spaces of analytic area-preserving diffeomorphisms of~$\BSi$ and of analytic vector fields on~$\BOm$ (endowed with the $C^\infty$ topology) are neither separable nor Baire. This failure, which also appears in analogous problems about residual properties of diffeomorphisms~\cite{Asaoka}, should be regarded as a sanity check of sorts. Heuristically, it shows that the heart of Theorem~\ref{T.main} is a nontrivial generic rigidity property of MHS equilibria, not abstract nonsense.
\end{remark}

\section{Genericity by means of the map~$\cJ$}
\label{S.cJ}

Our goal in this section is to show that generic properties of the set of braided fields $\cB(\overline\Om)$ can be studied in terms of generic properties of the diffeomorphisms $\diff$. This result turns out to be crucial for our approach because it allows us to import ideas and techniques from the theory of area-preserving diffeomorphisms of the disk. Let us start by proving Proposition~\ref{P.Pi}, which says that the Poincar\'e map of a braided field preserves a measure on the surface~$\Si$ that depends on the field itself:

\begin{proof}[Proof of Proposition~\ref{P.Pi}]
	As the swirl is positive (i.e., $B^\te>0$ in~$\BOm$), let us consider the vector field
	\[
	\tB := \frac1{B^\te(r,\te,z)}B= \frac{B^r(r,\te,z)}{B^\te(r,\te,z)}\partial_r+ \partial_\te+ \frac{B^z(r,\te,z)}{B^\te(r,\te,z)}\partial_z\,.
	\]
Since $\tB$ and~$B$ are proportional, both fields define the same Poincar\'e map on the transverse section~$\Si$.
Moreover, if $\tilde\phi^t$ denotes the time~$t$ flow of~$\tB$, it is easy to see that the Poincar\'e map of~$B$, $\Pi_B:\BSi\to\BSi$, is  given by the restriction to~$\BSi$ of $\tilde\phi^{2\pi}$.

Given a Borel set $V\subset\BSi$, we use the notation
\[
\mu_B(V)=\int_V B^\te(r,0,z)\, r\, dr\, dz
\]
for its $\mu_B$-measure and
\[
V_\de:=\{(r,\te,z): (r,z)\in V \text{ and } |\te|<\de/2\}\subset\BOm
\]
for its thickening of width~$\de\ll1$. Since $B^\te\tB$ is obviously divergence-free, the flow~$\tilde\phi^t$ must preserve the volume measure $B^\te(r,\te,z)\, r\, dr\, d\te\, dz$ on~$\BOm$. Note that, for small~$\de$, the invariant measure of the set~$V_\de$ is
\begin{align}
\int_{V_\de}	B^\te(r,\te,z)\, r\, dr\, d\te\, dz&= \int_{V_\de}	[B^\te(r,0,z)+O(\de)]\, r\, dr\, d\te\, dz\notag\\
&=\de \, \mu_B(V)+O(\de^2)\,.\label{E.measVde}
\end{align}
Furthermore, since $\tilde\phi^{2\pi}(r,\te,z)=\tilde\phi^{2\pi}(r,0,z)+O(\de)$ for all $(r,\te,z)\in V_\de$, we infer that
\begin{align*}
\int_{\tilde\phi^{2\pi}(V_\de)}	&B^\te(r,\te,z)\, r\, dr\, d\te\, dz= \int_{-\de/2}^{\de/2}\left(\int_{\Pi_B(V)}	[B^\te(r,0,z)\, r+O(\de)]\,  dr\,  dz\right)d\te\\
&=\de \, \mu_B(\Pi_B(V))+O(\de^2)\,.
\end{align*}	
Equating this quantity with~\eqref{E.measVde} and letting $\de\to0$ yields $\mu_B(\Pi(V))=\mu_B(V)$ for any Borel set~$V\subset\BSi$, as claimed.
\end{proof}

Next we prove Lemma~\ref{L.mapJ}, which concerns the existence of a continuous map $\cJ:\tcB\to\diff$ so that $\cJ(B)$ is conjugate to the Poincar\'e map of~$B$:

\begin{proof}[Proof of Lemma~\ref{L.mapJ}]
For any braided field $B\in\tcB$, consider the constant
	\[
	c_B:=\mu(\BSi)/\mu_B(\BSi)\,,
	\]
and recall that
\[
\mu_B=B^\theta(r,0,z)\,r\,dr\,dz\,.
\]

By~\cite{DM}, there exists a smooth diffeomorphism~$Y_B$ of~$\BSi$ which pulls the measure~$\mu$ into~$c_B\mu_B$:
	\[
	Y_B^*\mu= c_B \mu_B\,.
	\]
This diffeomorphism is nonunique. However, since the density of the volume measure $\mu_B$ depends continuously on~$B$, the method of proof in~\cite{DM} provides a continuous map $B\in\tcB\mapsto Y_B\in\Diff^+(\Si)$, endowed with their $C^\infty$~topologies.
This is because the diffeomorphism~$Y_B$ is obtained constructively by applying Steps~1 and~2 in the proof of~\cite[Theorem~1']{DM} to the equation
	\[
	\det \nabla Y_B = B^\theta(r,0,z)\,r\quad\text{in }\Si\,,\qquad Y_B=\mathrm{identity}\quad\text{on }\pd\Si\,,
	\]
	whose coefficients depend continuously on~$B$, and the solution is bounded as
	\[
	\|Y_B\|_{C^k(\Sigma)}+
	\|Y_B^{-1}\|_{C^k(\Sigma)}\leq C_k\left(\|B^\theta\|_{C^k(\Sigma)}
	+ \Big\|\frac{1}{B^\theta}\Big\|_{C^k(\Sigma)}\right)
	\]
	for all~$k$. Setting
$$\cJ(B):=Y_B\circ\Pi_B\circ Y_B^{-1}\,,$$
it readily follows by construction that $\cJ(B)\in\diff$,
and using the well known fact~\cite[Chapter~3.1]{Palis} that the Poincar\'e map also depends continuously on~$B$, we finally infer that the map~$\cJ$ is continuous.
\end{proof}

The central result of this section is Lemma~\ref{L.residual}, which ensures that the set $\cJ^{-1}(\cU)$ is locally residual if~$\cU$ is locally residual. To establish this, we prove that the map $\cJ$ is open and onto.

\begin{proof}[Proof of Lemma~\ref{L.residual}]
The first observation is that the map $\cJ:\tcB\to\diff$ is onto. Notice that $\mu= r\, dr\, dz$ and that using the coordinates~$(r,\te,z)$ we can identify $\Om$ with $\Si\times (\RR/2\pi\ZZ)$ and the Euclidean volume with $\mu\,d\te$. Since the braided field $\partial_\te\in\cB(\overline\Om)$ yields a Poincar\'e map that is the identity, and the group of diffeomorphisms $\diff$ is connected by $C^\infty$ arcs~\cite{Sauli}, the surjectivity follows from a straightforward application of a theorem of Treschev~\cite{Treschev}, who showed that given a diffeomorphism $F\in \diff$, there exists a divergence-free braided field $B\in \tcB$ whose Poincar\'e map is $\Pi_B=F$ and $\mu_B=\mu$, so that $\cJ(B)=\Pi_B=F$.

Therefore, $\cJ:\tcB\to \diff$ is an onto continuous map by Lemma~\ref{L.mapJ}. To show it is also open, we argue as follows. As the spaces~$\tcB$ and $\diff$ are metrizable, let us respectively denote by $d_{\cB}$ and~$d$ the distance functions that define the smooth topology, e.g.,
	\begin{align*}
	d_{\cB}(B,B')&:=\sum_{k=0}^\infty 2^{-k}\min\{\|B-B'\|_{C^k(\BOm)},1\}\,,\\
	d(F,F') &:=\sum_{k=0}^\infty 2^{-k}\min\{\|F-F'\|_{C^k(\BSi)},1\}\,.
	\end{align*}
	
Suppose that $\cJ$ is not an open map, so there exists an open subset $U\subset\tcB$ and a diffeomorphism $F_0=\cJ(B_0)$ with $B_0\in U$ such that, for all $n\geq1$, there is some $F_n\in\diff\backslash \cJ(U)$ such that $d(F_0,F_n)<1/n$. Taking the Poincar\'e map $\Pi_{B_0}=Y_{B_0}^{-1}\circ F_0\circ Y_{B_0}$, where the smooth diffeomorphism $Y_{B_0}$ of $\Si$ was introduced in Lemma~\ref{L.mapJ}, we set the maps
\[
\Pi_n:=Y_{B_0}^{-1}\circ F_n\circ Y_{B_0}\,,
\]
that are diffeomorphisms in $\Diff^+_{\mu_{B_0}}(\BSi)$. As, just as before, $\Diff^+_{\mu_{B_0}}(\BSi)$ is connected by $C^\infty$ arcs, the aforementioned Treschev's theorem~\cite{Treschev}  then implies that for all $n\geq0$ there exists a divergence-free braided field $B_n\in \tcB$ whose Poincar\'e map is $\Pi_{B_n}=\Pi_n$ and the invariant area measure is $\mu_{B_n}=\mu_{B_0}$. Moreover, Treschev's suspension map assigning a vector field~$B_n$ to each~$\Pi_n$ is continuous~\cite[Corollary]{Treschev}, so the distance between the fields can be estimated as
	\begin{equation}\label{e.close}
	d_{\cB}(B_n,B_0)\leq C/n\,.
	\end{equation}
Accordingly, for large enough $n$, the braided field $B_n$ lies in $U$ and
$$\cJ(B_n)=Y_{B_0}\circ\Pi_{n}\circ Y_{B_0}^{-1}=F_{n}\,,$$
where we have used that the conjugacy maps $Y_{B_n}=Y_{B_0}$ for all $n\geq0$ because $\mu_{B_n}=\mu_{B_0}$. This means that $F_n\in\cJ(U)$, which is a contradiction. We thus conclude that the map $\cJ$ is open.

Since we have shown that the continuous map $\cJ$ is open and onto, it is then standard that if a set $\cV'$ is dense in an open set $\cV\subset\diff$, the preimage
\[
\cJ^{-1}(\cV'):=\{B\in\tcB: \cJ(B)\in\cV'\}
\]
is dense in the open set $\cJ^{-1}(\cV)\subset \tcB$. Since residual sets are countable intersections of sets whose interior is dense, the lemma follows.
\end{proof}

\section{Generic non-integrability: proof of Theorem~\ref{T.cU}}
\label{S.non-integrable}

In this section we prove Theorem~\ref{T.cU}, which ensures that a residual set of $\mu$-preserving diffeomorphism of~$\Si$ are non-integrable and Morse--Smale on the boundary. We present two proofs but the first one is more illustrative from the dynamical viewpoint.

First we show that the set
\[
\MS:=\{ F\in\diff: F|_{\pd\Si} \text{ is a Morse--Smale diffeomorphism of }\pd\Si\}
\]
is open and dense in $\diff$. To see this, we shall use the following auxiliary result:

\begin{lemma}[\cite{Tsuboi}, Lemma 2.2]\label{L.Tsuboi}
The restriction to the boundary defines an onto and open continuous map $\tr: \diff\to \Diff^+(\pd\Si)$, where $\Diff^+(\pd\Si)$ denotes the space of $C^\infty$ diffeomorphisms of~$\pd\Si$ that are connected with the identity.
\end{lemma}

Therefore, since the set of Morse--Smale diffeomorphisms of~$\pd\Si\cong \mathbb S^1$ is open and dense in the space of $C^\infty$~diffeomorphisms \cite[Chapter~4.4]{Palis} and the boundary trace map $\tr: \diff\to \Diff^+(\pd\Si)$ is onto, open and continuous, it is standard that
\[
\MS:=\tr^{-1}(\{f\in \Diff^+(\pd\Si): f \text{ is Morse--Smale}\})
\]
is also open and dense.

Our objective in the rest of the proof is to show that the set $\cUni$ of non-integrable diffeomorphisms whose boundary restriction is Morse--Smale
is residual in $\diff$. To this end it is convenient to introduce the following sets of $\mu$-preserving diffeomorphisms:
\begin{align*}
&\mathcal U_{\text{ND}}:=\{ F\in\diff:\text{ all the periodic points of }F \text{ are nondegenerate}\}\,,\\
&\mathcal U_{\text{PD}}:= \{ F\in\diff:\text{ the set of periodic points of }F \text{ is dense on }\Sigma\}\,.
\end{align*}
Recall that an $N$-periodic point of a diffeomorphism is {\em nondegenerate} if it is either elliptic or hyperbolic; in particular, it is isolated in the set of $N$-periodic points.

It turns out that both sets of diffeomorphisms are generic. The fact that $\mathcal U_{\text{ND}}$ is residual in $\diff$ was proved by Robinson~\cite[Theorem~1B]{Robinson}, while an analogous result for the set $\mathcal U_{\text{PD}}$ was proved by Pirnapasov and Prasad in~\cite{Prasad}. Accordingly, the space of diffeomorphisms whose periodic set consists of nondegenerate periodic points that are dense on $\Sigma$, i.e., $\mathcal U_{\text{ND}}\cap \mathcal U_{\text{PD}}$, is residual in $\diff$. We claim that if $F\in\mathcal U_{\text{ND}}\cap \mathcal U_{\text{PD}}$ then $F$ is non-integrable:

\begin{proposition}\label{P.aa}
$\mathcal U_{\mathrm{ND}}\cap \mathcal U_{\mathrm{PD}}\cap\MS\subset \mathcal \cUni$.	
\end{proposition}

\begin{proof}
Take $F\in \mathcal U_{\mathrm{ND}}\cap \mathcal U_{\mathrm{PD}}\cap\MS$ and assume that is has a nonconstant smooth first integral, that is, a nonconstant function $h\in C^\infty(\Si)$
such that $h\circ F= h$. As $F|_{\pd\Si}$ is a Morse--Smale diffeomorphism of~$\pd\Si$, then $F|_{\pd\Si}$ does not admit any nonconstant first integral, so we easily infer that the function $h$ must be constant on the boundary (i.e., $h|_{\pd\Si}=c_0\in\RR$). Since~$h$ is not constant, by Sard's theorem it has a regular value $c\neq c_0$. Let $\ga$ be a connected component of the regular level set $h^{-1}(c)$. This level set cannot intersect~$\pd\Si$, so $\ga$ must be a closed curve (i.e., diffeomorphic to~$\SS^1$) that is contained in~$\Si$. Then it follows from~\cite[Example 5.3]{Arnold} that there exists a neighborhood $U\subset\Si$ of the curve~$\ga$ and (local) action-angle coordinates, that is, a $C^\infty$ diffeomorphism
	\begin{equation}\label{E.The}
	\Theta_F: (-R_0,R_0)\times (\RR/2\pi\ZZ)\to U\,,
	\end{equation}
	where $R_0>0$, in which~$F$ and the area measure read as
	\begin{equation}\label{E.The2}
	(\Theta_F^{-1}\circ F\circ\Theta_F)(R,\vp)= (R, \vp+\om_F(R))\,,\qquad \Theta_F^*\mu= dR\, d\vp\,,
	\end{equation}
for some frequency function $\om_F\in C^\infty((-R_0,R_0))$. If $\om_F$ is not a constant or if it is a constant that is a rational multiple of $2\pi$, we deduce that $F$ exhibits invariant closed curves filled by periodic points, which contradicts the fact that $F\in  \mathcal U_{\text{ND}}$. Therefore, $\om_F=2\pi C$ for some irrational number $C$, which in view of Equation~\eqref{E.The2} implies that all the invariant curves of $F$ are quasi-periodic, so in particular $F$ does not have any periodic points in the invariant set $U$. Since this contradicts the fact that $F\in \mathcal U_{\text{PD}}$, we finally infer that such a nonconstant first integral $h$ cannot exist, and hence $F\in \mathcal U_{\text{NI}}$ as we wanted to show.
\end{proof}

Finally, since the sets $\mathcal U_{\text{ND}}$, $\mathcal U_{\text{PD}}$ and $\MS$ are residual, their intersection $\mathcal U_{\text{ND}}\cap \mathcal U_{\text{PD}}\cap\MS$ is residual as well, which in turn implies that $\mathcal U_{\text{NI}}$ is also residual by Proposition~\ref{P.aa}. This completes the proof of Theorem~\ref{T.cU}.

\subsection{An alternative proof}

For the benefit of the reader we shall next provide an alternative proof of Theorem~\ref{T.cU}, based on classical ideas introduced by Markus and Meyer~\cite{MM74}, which does not use the set $\mathcal U_{\text{PD}}$ and the recent (and quite sophisticated) results of~\cite{Prasad}. Nevertheless, a major advantage of the less elementary but more direct proof given above is that it reveals that the key dynamical obstruction to integrability in the $C^\infty$ topology is the existence of a dense set of nondegenerate periodic points.

First, we define the set of integrable diffeomorphisms as
\begin{align*}
\cUi:=&\{ F\in\diff: F \text{ has a nonconstant first integral }h\in C^\infty(\Si)\}\,.
\end{align*}
Notice that the proof of Proposition~\ref{P.aa} shows the existence of local action-angle variables on some annular domain $U_F\subset\Sigma$ for any $F\in \cUi\cap\MS$ (of course, in general $U_F$ is not unique). We denote by $\NA$ the set of diffeomorphisms $F\in \cUi\cap\MS$ for which the frequency function $\om_F$ associated with the action-angle variables on $U_F$ is nonconstant.

It is easy to see that if $\NA$ is residual in $\cUi\cap\MS$, then $\cUni$ is residual too. Indeed, in this case, as
\[
(\diff\backslash \NA)\cap\MS=\cUni \cup\Big((\cUi\cap\MS)\backslash\NA\Big)\,,
\]
and since we showed in the proof of Proposition~\ref{P.aa} that
$$\mathcal U_{\text{ND}}\cap\MS\subset (\diff\backslash \NA)\cap\MS\,,$$
the fact that $\mathcal U_{\text{ND}}\cap\MS$ is residual (by a classical result of Robinson~\cite{Robinson}) and the assumption that the set of non-twist integrable diffeomorphisms $(\cUi\cap\MS)\backslash\NA$ is meagre immediately imply that $\cUni$ is residual as well.

Therefore, to complete the proof of Theorem~\ref{T.cU}, it suffices to show that $\NA$ is residual in $\cUi\cap\MS$. To this end, we observe that, taking a countable basis of the topology of $\Sigma$ we can fix a countable set of disks $\{B^{(k)}\}_{k=1}^\infty\subset\Sigma$ such that for any $F\in \cUi\cap\MS$ each annular component of the corresponding set $U_F$ where the action-angle variables are defined contains the domain $B^{(k)}$ for some $k$. We denote by
$$\cUi^{(k)}\subset\cUi$$
the set of integrable diffeomorphisms that admit action-angle variables in an annular set that contains $B^{(k)}$, and
$$\NA^{(k)}\subset\cUi^{(k)}$$
those diffeomorphisms whose corresponding frequency function on the aforementioned annular set is nonconstant.

If $F\in \cUi^{(k)}\cap\MS$ for some $k$, then $F$ takes the form
\[
F(R,\vp)=: (R, \vp+\om_F(R))
\]
in action-angle variables $(R,\vp)\in(-R_0,R_0)\times (\RR/2\pi\ZZ)$ on some domain $U^{(k)}_F\supset B^{(k)}$, and the invariant area measure is $dR\,d\vp$. If $\om_F$ is not a constant function, then $F\in\NA^{(k)}$. If $\om_F$ is constant, we introduce the diffeomorphism $F'\in \cUi^{(k)}\cap\MS$ defined as $F':=F$ in $\Si\backslash U^{(k)}_F$ and
\[
F'(R,\vp):=(R,\vp+\om_F+\ep E(R))
\]
in action-angle variables of $U^{(k)}_F$, where $\ep>0$ and where $E(R)$ is any nonconstant $C^\infty$~function supported in $(-R_0/2,R_0/2)$. Obviously, $F'\in \NA^{(k)}$. Since $\ep$ can be taken as small as desired, this shows that $\NA^{(k)}$ is dense in $\cUi^{(k)}\cap\MS$.

Next we show that $\NA^{(k)}$ is also open. Assume that $F,F'\in \cUi^{(k)}$ for some $k$, and $F\in \NA^{(k)}$. If $F$ and $F'$ are $C^\infty$-close to each other, Moser's twist theorem for area-preserving diffeomorphisms of the disk~\cite{Moser2} implies that near any invariant circle $\{R=R_1\}$ of $F$ with $\om_F'(R_1)\neq0$, the diffeomorphism $F'$ exhibits a positive measure set of invariant circles close to those of $F$ and with the same frequency. Since $\om_F$ is nonconstant, and both action-angle domains $U^{(k)}_F$ and $U^{(k)}_{F'}$ contain the same disk $B^{(k)}$, it necessarily follows that $\om_{F'}$ is not constant either, thus showing that $F'\in\NA^{(k)}$.

Summarizing, the previous arguments show that $(\cUi^{(k)}\cap\MS)\backslash\NA^{(k)}$ is meagre for all $k\in\NN$. It is obvious by definition that
\[
(\cUi\cap\MS)\backslash\NA\subset\bigcup_{k=1}^\infty (\cUi^{(k)}\cap\MS)\backslash\NA^{(k)}\,,
\]
so the property that a countable union of meagre sets is meagre, then yields that $(\cUi\cap\MS)\backslash\NA$ is a meagre set, as we wanted to show.

\section{A rigidity result for generic relaxation: proof of Theorem~\ref{T.rigidity}}
\label{S.rigidity}

Our objective in this section is to prove the rigidity result stated in Theorem~\ref{T.rigidity}. For this, let us now assume that $B\in \cJ^{-1}(\cUni )$ and that there exists a smooth volume-preserving diffeomorphism~$Y$ of $\overline\Omega$ such that $\lB:=Y_*B$ is an MHS equilibrium, where $\cUni $ was defined in Theorem~\ref{T.cU}. An elementary observation, which follows directly from the assumption that $\cJ(B)\in\cUni $ and which we will use repeatedly, is the following:

\begin{lemma}\label{L.fi}
	The field~$\lB$ does not admit any smooth first integrals. That is, if a function~$g\in C^\infty(\BOm)$ satisfies $\lB\cdot\nabla g=0$, then $g$ is constant.
\end{lemma}

\begin{proof}
	
Since $\lB=Y_* B$, the disk-like closed submanifold $\Sigma':=Y(\Sigma)$ is a global transverse section of $\lB$, which then defines a return (or Poincar\'e) map $\Pi_\lB:\Sigma' \to\Sigma'$. It is clear that $\Pi_B$ and $\Pi_\lB$ are conjugate:
\[
\Pi_\lB=Y\circ \Pi_B\circ Y^{-1}
\]
and then by Equation~\eqref{E.cJYB}, $\Pi_\lB$ and $\cJ(B)$ are related by
	\begin{equation}\label{E.PilB}
	\Pi_{\lB}:= Y\circ Y_B^{-1}\circ\cJ(B)\circ Y_B\circ Y^{-1}\,.
	\end{equation}

If $g\in C^\infty(\BOm)$ is a first integral of~$\lB$, necessarily $h':=g|_{\Si'}\in C^\infty(\Sigma')$ is invariant under~$\Pi_{\lB}$, i.e., $h'\circ \Pi_\lB=h'$. This implies that
$$h:= h'\circ Y\circ Y^{-1}_B\in C^\infty(\Si)$$
satisfies $h\circ\cJ(B)= h$. As~$\cJ(B)\in\cUni $, from condition~(ii) in Theorem~\ref{T.cU} we obtain that $h$, and hence $h'$, must be constant. Since $\Sigma'$ is a global transverse section of $\lB$ we finally conclude that $g$ is constant on $\overline{\Om}$, as we wanted to show.
\end{proof}

The well known dichotomy that MHS equilibria are either Beltrami fields (with constant proportionality factor) or they are integrable, then yields this easy lemma:

\begin{lemma}\label{L.existela}
	There exists a real constant~$\la$ such that\smallskip
	\begin{gather}\label{E.existela}
		\begin{split}
\Div \lB=0\quad \text{and}\quad \curl \lB=\la \lB\qquad \text{in }\Om\,,\\
\lB\cdot N=0 	\qquad \text{on }\pd \Om\,.		
		\end{split}
	\end{gather}
\end{lemma}

\begin{proof}
By assumption $\lB$ is an MHS equilibrium, so there is a pressure function $P\in C^\infty(\overline\Om)$ such that	
\[
\lB \times \curl \lB +\nabla P=0\,, \qquad \Div \lB=0\,.
\]
In particular, $\lB\cdot\nabla P=0$, so $P$ is a first integral of $\lB$, which is then constant by Lemma~\ref{L.fi}. Therefore, $\lB\times\curl\lB=0$. Since $\lB=Y_*B$ does not vanish because $B\neq0$, this implies that
\[
\curl\lB= f\lB\,,
\]
with $f:= |\lB|^{-2}\lB\cdot\curl\lB\in C^\infty(\BOm)$. Taking the divergence of this equation and using that $\lB$ is divergence-free, we conclude that $\lB\cdot \nabla f=0$, so $f$ is a first integral of $\lB$. By Lemma~\ref{L.fi} again, then $f=\la$ for some constant~$\la$. The boundary condition $\lB\cdot N=0$ on~$\pd\Om$ easily stems from the facts that $Y$ maps $\pd\Om$ onto itself and $B\cdot N=0$.
\end{proof}

In fact, the constant factor $\lambda$ in Equation~\eqref{E.existela} cannot take any real value, but it must be an eigenvalue of the curl operator in $\Om$, that is, $\la=\la_j$ for some positive integer $j$, where the set $\{\la_j\}_{j=1}^\infty$ was introduced in Section~\ref{SS.NI}. This is the content of the following lemma:

\begin{lemma}\label{L.laeigen}
$\la=\la_j$ for some integer $j\geq1$.
\end{lemma}

\begin{proof}
In terms of the components of~$\lB$ and of the normal vector~$N$ in cylindrical coordinates,
\begin{align*}
\lB&=: \lB^r(r,\te,z)\, \partial_r + \lB^\te (r,\te,z)\, \partial_\te+ \lB^z(r,\te,z)\, \partial_z\,,\\
N&=: N^r(r,z)\,\partial_r+ N^z(r,z)\, \partial_z\,,
\end{align*}
the system~\eqref{E.existela} reads as
\begin{align}\label{E.system}
		\begin{split}
	\pd_r(r \lB^r)+r\pd_\te \lB^\te+r\pd_z \lB^z&=0\,,\\
	\pd_\te \lB^z-r^2\pd_z\lB^\te - \la r\lB^r&=0\,,\\
	\pd_z\lB^r-\pd_r \lB^z- \la r\lB^\te&=0\,,\\
	\pd_r(r^2\lB^\te)-\pd_\te \lB^r-\la r \lB^z&=0\,,
	\end{split}
\end{align}
with the boundary condition
\[
\lB^r(r,\te,z)\, N^r(r,z)+ \lB^z(r,\te,z)\, N^z(r,z)=0\qquad \text{on }\pd\Om\,.
\]

As the coefficients of this system and the functions $N^r,N^z$ are independent of~$\te\in\RR/2\pi\ZZ $, we can differentiate the system and the boundary condition with respect to this variable to conclude that the vector field
\[
w:= \pd_\te \lB^r(r,\te,z)\, \partial_r + \pd_\te\lB^\te (r,\te,z)\, \partial_\te+ \pd_\te\lB^z(r,\te,z)\, \partial_z
\]
satisfies the equation
\[
\curl w=\la w\quad \text{in }\Om\,,\qquad w\cdot N=0\quad \text{on }\pd \Om\,.
\]
Furthermore, $w$ is $L^2$-orthogonal to the space of harmonic fields because one can write, using~\eqref{E.defhOm} and the $2\pi$-periodicity in~$\te$ of $\lB$,
\[
\int_{\Om} w\cdot h_\Om\, dx= \int_{\Si}\left( \int_0^{2\pi}\pd_\te \lB^\te (r,\te,z)\, d\te\right)rdr\, dz=0\,.
\]

Suppose that $\la$ is not an eigenvalue. Then we conclude (see Equation~\eqref{E.eigen}) that $w=0$, which means that the field $\lB$ is axisymmetric, i.e., the functions $\lB^r, \lB^\theta, \lB^z$ do not depend on~$\theta$.

It is then easy to see that one can use the first equation of the system~\eqref{E.system} to write the field~$\lB$ in terms of a smooth scalar function $\psi(r,z)$ and the swirl $\lB^\te(r,z)$ as
\[
\lB=\frac1r\left[\pd_z\psi(r,z) \,\partial_r -\pd_r\psi(r,z)\, \partial_z \right]+\lB^\te(r,z)\, \partial_\te\,.
\]
It is obvious that $\psi$ is a first integral of $\lB$, that is, $\lB\cdot\nabla\psi=0$.

By assumption, $\cJ(B)\in\cUni $, so Lemma~\ref{L.fi} then ensures that~$\psi$ is constant, so in fact
\[
\lB= \lB^\te(r,z)\,\partial_\te
\]
with $\lB^\te\neq0$. Then the Poincar\'e map~$\Pi_{\lB}$, and therefore $\cJ(B)$ by the conjugacy~\eqref{E.PilB}, are the identity. This contradicts the assumption that $\cJ(B)$ is non-integrable (condition~(ii) in Theorem~\ref{T.cU}), and the lemma then follows.
\end{proof}

We are now ready to finish the proof of Theorem~\ref{T.rigidity}. By Lemma~\ref{L.laeigen}, we know that there exists some eigenvalue~$\la_j$ such that the vector field $\lB$ satisfies Equation~\eqref{E.existela} with $\la=\la_j$. To take into account the harmonic part of~$\lB$, let us define the quantity
\[
c_0:=\int_{\Om} h_\Om\cdot \lB\, dx\,,
\]
and using the vector field $V_j$ defined in Section~\ref{SS.NI}, we note that
\[
v:=\lB- c_0 V_j
\]
satisfies the equation
\begin{gather*}
\curl v=\la_j v\quad \text{in }\Om\,,\qquad v\cdot N=0\quad \text{on } \pd \Om\,,\qquad \int_{\Om} v\cdot h_\Om\, dx=0\,,
\end{gather*}
so it is an eigenfield with eigenvalue~$\la_j$ by Equation~\eqref{E.eigen}. This yields the expression~\eqref{E.formu}, so Theorem~\ref{T.rigidity} is proven.

\section*{Acknowledgements}

The authors are grateful to Simon Candelaresi, Gunnar Hornig, Stuart Hudson, David MacTaggart and Anthony Yeates for their comments and suggestions on plasma physics. We are indebted to Masayuki Asaoka and Pierre Berger for their explanations concerning the Newhouse domain. This work has received funding from the European Research Council (ERC) under the European Union's Horizon 2020 research and innovation programme through the grant agreement~862342 (A.E.). It is partially supported by the grants CEX2019-000904-S and PID2019-106715GB GB-C21 (D.P.-S.) funded by MCIN/AEI/10.13039/501100011033, and Ayudas Fundaci\'on BBVA a Proyectos de Investigaci\'on Cient\'ifica 2021 (D.P.-S.).

\appendix

\section{Fields with Reeb components}
\label{A.Reeb}

In this appendix, we recall the precise statement and proof of Cieliebak and Volkov's result~\cite{Cieliebak} that MHS equilibria cannot exhibit a Reeb component (see also~\cite{PRT} for a significant extension of this result using the theory of plugs). We recall that a {\em Reeb component}\/ of a vector field $\lB$ is a smooth compact embedded surface $\cA$, diffeomorphic to a cylinder, and invariant under the flow $X^t$ of~$\lB $, such that:
\begin{enumerate}
\item $\pd\cA$ consists of two periodic orbits of $\lB$ that induce the boundary orientation of~$\pd\cA$.
\item $X^t(x)$ converges to $\pd\cA$ as $t\to\pm\infty$ for all interior points $x$ of $\cA$.
\end{enumerate}
The observation of Cieliebak and Volkov is based on an application of Stokes' theorem:

\begin{lemma}[\cite{Cieliebak}, Lemma 2.3]\label{T.Cielieback}
	Assume that $\lB $ is a smooth MHS equilibrium in an open set $U\subset\RR^3$. Then~$\lB $ cannot have any Reeb components in $U$.
\end{lemma}

\begin{proof}
Suppose that $\cA\subset U$ is a Reeb component of $\lB$. By assumption, $\lB$ satisfies the MHS equations
\[
\lB \times \curl \lB+\nabla P=0\,, \qquad \Div \lB=0\,,
\]
for some scalar function $P\in C^\infty(U)$. It is elementary to check that $P$ is a first integral of~$\lB $, so the assumption that $\cA$ is invariant under the flow of $\lB$ implies that $P|_{\cA}$ is a first integral of $\lB|_{\cA}$. Since the dynamics of a Reeb component does not admit nontrivial first integrals, we infer that $P|_{\cA}$ is a constant. Moreover, since $\lB |_{\cA}$ is nonzero and tangent to~$\cA$, we easily conclude from the MHS equation that
$$\curl \lB \cdot N_\cA=0$$
on~$\cA$, where $N_\cA$ is a unit normal vector of the surface~$\cA$. Using Stokes' formula, we then obtain
\[
0=\int_{\cA} \curl \lB \cdot N_\cA \, d\si = \int_{\pd\cA}\lB \cdot d\ell\,.
\]
But the fact that $\lB $ induces the boundary orientation on~$\pd\cA$ ensures that the above line integral must be strictly positive, which is a contradiction.
\end{proof}

By a simple scaling argument, Lemma~\ref{T.Cielieback} implies that, on any domain of $\RR^3$, there is a set of divergence-free vector fields that are not topologically equivalent to an MHS equilibrium that is dense in the Sobolev space $H^s$ for any $s<\frac32$:

\begin{corollary}\label{C.Reeb}
	Consider a smooth divergence-free vector field $B_0$, defined on an open set~$U\subset\RR^3$, take a point~$x^0$ where $B_0$ does not vanish and any real number $0\leq s<\frac 32$. Then, for any $\de>0$, there exists another smooth divergence-free vector field $\tB_0$ on~$U$ which has a Reeb component, coincides with~$B_0$ outside the ball centered at~$x^0$ of radius~$\de $, and approximates~$B_0$ as
	\[
	\|B_0-\tB_0\|_{H^{s}(U)}<\de\,.
	\]
	The field~$\tB_0$ is not topologically equivalent to an MHS equilibrium.
\end{corollary}

\begin{proof}
Consider the affine change of variables $\phi_\ep (y):= x^0 + \ep y$, where $\ep>0$ is a small parameter. It maps the ball $\mathbb{B}_2$ centered at the origin of radius~2 onto the ball centered at~$x^0$ of radius~$2\ep$. Now consider the push-forward $(\phi_\ep^{-1})_* (B_0)$, which is still a divergence-free vector field in the $y$-coordinates. Note that, on the ball $\mathbb{B}_2$, $(\phi_\ep^{-1})_* (B_0)$ is $\ep$-close to the constant field $B_0(x^0)\neq 0$ in the smooth topology, in the sense that the $C^\infty(\mathbb{B}_2)$-distance between these fields is at most $C_1\ep$ for some fixed constant~$C_1$. Hence one can then use a Wilson-type plug construction in the volume-preserving setting~\cite{Kup} to obtain another divergence-free $C^\infty$ vector field $B_0^\ep$ which coincides with $(\phi_\ep^{-1})_* (B_0)$ outside the ball $\mathbb{B}_1$ of radius~1 and which features a Reeb component in $\mathbb{B}_1$. The $\ep$-closeness to a constant vector field easily implies the uniform bound
$$\|B_0^\ep- (\phi_\ep^{-1})_* (B_0)\|_{C^2(\mathbb{B}_1)}<C_2$$
for all small~$\ep$ and some $\ep$-independent constant $C_2$. If we now take the pushed-forward vector field $\tB_0^\ep:= (\phi_\ep)_*(B_0^\ep)$, which is divergence-free and coincides with $B_0$ in the complement of the ball of radius $\ep$ centered at $x^0$, for any $s<\frac32$ one can estimate
\begin{align*}
\|B_0-\tB_0^\ep \|_{H^{s}(U)}^2&=	\|B_0-\tB_0^\ep \|_{L^2(\phi_\ep(\mathbb{B}_1))}^2+ \||\nabla|^s(B_0-\tB_0^\ep) \|_{L^2(\phi_\ep(\mathbb{B}_1))}^2\\
&\leq C\ep^3  + C\ep^{3-2s}<\delta\,,
\end{align*}
if we pick $\ep$ sufficiently small (depending on~$\de $). Here $C$ depends on~$C_2$ but not on~$\ep$. The corollary then follows from Lemma~\ref{T.Cielieback} by setting $\tB_0:=\tB_0^{\ep}$ and noticing that this vector field exhibits a Reeb component in $\phi_{\ep}(\mathbb{B}_1)$.
\end{proof}

\begin{remark}
In the particular case that $U$ is a toroidal domain of $\RR^3$, it is clear that a vector field $B_0$ which exhibits a Reeb component cannot be braided. Accordingly, Corollary~\ref{C.Reeb} cannot be used to prove an analog of Theorem~\ref{T.main}, which is a statement about generic braided fields. Any construction that makes use of plugs~\cite{PRT} will have the same limitation, since plugs do not admit a global transverse section.
\end{remark}

\section{The Newhouse domain}
\label{A.Newhouse}

If $\bD$ is the unit disk of $\RR^2$ and $\mu_{\bD}$ is the standard area measure, we denote by $\Diff^+_{\mu_{\bD}}(\overline\bD)$ the group of smooth area-preserving diffeomorphisms of the disk that preserve the orientation. In this article we have been mainly concerned with the group $\diff$ of $\mu$-preserving diffeomorphisms connected with the identity of the disk-like compact surface $\Sigma$, but in view of the Dacorogna--Moser lemma~\cite{DM}, there is no actual difference between $\Diff^+_{\mu_{\bD}}(\overline\bD)$ and $\diff$:

\begin{lemma}\label{L.SitoD}
There exists a $C^\infty$ diffeomorphism $\Phi: \overline\bD\to \BSi$ such that $\Phi^*\mu=c\mu_{\bD}$ for some constant $c>0$. The conjugation $F\mapsto \Phi^{-1}\circ F\circ \Phi$ defines a group isomorphism $\diff\to \Diff^+_{\mu_{\bD}}(\overline\bD)$.
\end{lemma}

This elementary lemma ensures that we can apply any result from the theory of area-preserving diffeomorphims of the disk in our context of $\diff$. In particular, a landmark in the study of conservative dynamical systems is the Duarte--Newhouse theorem~\cite{Newhouse,Duarte} showing that the so-called Newhouse domain of the group $\Diff^+_{\mu_{\bD}}(\overline\bD)$ is a non-empty open set. Recall that the {\em Newhouse domain}\/ is defined as the set $\cN\subset\Diff^+_{\mu_{\bD}}(\overline\bD)$ which is the interior of the closure in the smooth topology of the diffeomorphisms exhibiting a hyperbolic periodic point with a homoclinic tangency.

The dynamics of the diffeomorphisms in the Newhouse domain is extremely rich and wild~\cite{GST,Tura15,Berger}. In particular, there is a residual subset of $\cN$ that exhibits fast growth of periodic points and universal dynamics (i.e., maps whose sets of renormalized iterations approximate every
possible dynamics arbitrarily well). The Newhouse domain is dense in the group $\Diff^1_{\mu_{\bD}}(\overline\bD)$ of $C^1$ area-preserving diffeomorphisms of the disk that preserve the orientation; this is a consequence of a theorem of Newhouse~\cite{Newhouse2} and of the fact that there are no Anosov diffeomorphisms of the disk. It is conjectured~\cite{Tura15} that an analogous result holds true in the $C^r$ topology, $r\geq2$, so that in particular the Newhouse domain is dense in $\Diff^+_{\mu_{\bD}}(\overline\bD)$. This conjecture remains wide open. The reader may consult~\cite{Turaev,Berger} for an account on the subject. As an aside remark, we mention that it has been proved recently that the Newhouse domain is nonempty for the space of Beltrami fields in $\RR^3$~\cite{BFP}.

\section{Ideal compressible MHD relaxation}
A non-resistive compressible plasma is described by the MHD~system
\begin{gather}
\pd_t\rho+\Div(\rho v)=0\,,\notag\\
	\pd_t B= \curl(v\times B)\,,\label{E.magdyn2}\\
	\pd_t (\rho v) + \Div(\rho v\otimes v)=\nu\De v+ B\times\curl B+\nabla P\,,\notag\\
	\Div B=0\,, \notag
\end{gather}
where $\rho$ is the plasma mass density (a generally non-constant positive function). The compressible MHD relaxation process is based on the observation that Equation~\eqref{E.magdyn2} can be written as a transport equation:

\begin{lemma}\label{L.compre}
Let $(\rho, v,B,P)$ be a smooth solution of the ideal compressible MHD equations. Then
\[
\partial_t\Big(\frac{B}{\rho}\Big)=\frac{B}{\rho}\cdot \nabla v - v\cdot \nabla \frac{B}{\rho}=:\Big[\frac{B}{\rho},v\Big]\,.
\]
In particular, if $X^t$ denotes the time $t$ flow of the vector field~$v$, this equation means that $\frac{B(x,t)}{\rho(x,t)}$ can be written in terms of the initial data $B_0(x),\rho_0(x)$ via the push-forward of the diffeomorphism $X^t$ (which is no longer volume-preserving):
\[
\frac{B(t,x)}{\rho(x,t)}=\frac{X^t_* B_0(x)}{\rho_0\circ (X^t)^{-1}}\,.
\]
\end{lemma}
\begin{proof}
It follows from an elementary calculation and the first, the second and the fourth equations in the MHD system:
\begin{align*}
\partial_t\Big(\frac{B}{\rho}\Big)&=\frac{1}{\rho}\curl(v\times B)+\frac{B}{\rho^2}\Div(\rho v)\\
&= \frac{1}{\rho}[B,v]+\frac{v\cdot\nabla\rho}{\rho^2}B\\
&=\Big[\frac{B}{\rho},v\Big]\,.
\end{align*}
Here, $[\cdot,\cdot]$ denotes the Lie bracket of vector fields and we have used the well known identity for vector fields $X,Y$,
\[
\curl(X\times Y)=(\Div Y)X-(\Div X)Y+[Y,X]\,.
\]
\end{proof}

In view of Lemma~\ref{L.compre}, the ideal compressible MHD relaxation consists in transporting the magnetic field $B_0$, up to a positive proportionality factor, via a diffeomorphism that is not volume-preserving in general. In case that this process yields a ``well-behaved'' smooth limit $\lB$, it would be an MHS equilibrium in the plasma domain that is orbitally equivalent to the divergence-free vector field $B_0$. By ``orbital equivalence'', we mean that there exists a $C^\infty$ diffeomorphism $X$ (not necessarily volume-preserving) of the corresponding domain and a $C^\infty$ positive factor $g$ such that
$$\lB= gX_*B_0\,.$$

Since the Poincar\'e map of a vector field is invariant under multiplication of the field by a positive factor, and the obstructions we found for topological relaxation (Morse-Smale on the boundary, non-integrability and growth of hyperbolic periodic points) are invariant under arbitrary diffeomorphisms (not only those that are volume-preserving), we immediately infer that an analogue of Theorem~\ref{T.main} holds in the context of ideal compressible MHD relaxation. More precisely, using the same notation as in the statement of Theorem~\ref{T.main}, we have:

\begin{theorem}\label{T.main2}
Let $\Om\subset\RR^3$ be an analytic axisymmetric toroidal domain. The subset of braided fields that are not orbitally equivalent to any smooth magnetohydrostatic equilibrium on~$\Om$ is locally generic in~$\tcB$.
	
More precisely, let $B_0\in \tcB$ be a vector field whose Poincar\'e map $\Pi_{B_0}:\BSi\to\BSi$ satisfies the following conditions:
	\begin{enumerate}
		\item The restriction $\Pi_{B_0}|_{\pd\Si}$ is a Morse--Smale diffeomorphism of~$\pd\Si$.
		\item $\Pi_{B_0}$ is non-integrable. That is, there does not exist a nonconstant function $h\in C^\infty(\BSi)$ such that $h\circ \Pi_{B_0}=h$.
		\item For all $j,k\geq1$,
		\[
		\limsup_{n\to\infty} \frac{\#\Per(\Pi_{B_0},n)}{N_{j,k}(n)}=\infty\,,
		\]
		where $\#\Per(\Pi_{B_0},n)$ is the number of hyperbolic periodic points of the Poincar\'e map $\Pi_{B_0}$ whose least period is~$n$ and where $\{N_{j,k}(n)\}_{j,k,n=1}^\infty$ is certain fixed countable set of positive integers.
	\end{enumerate}
	Then there does not exist a smooth diffeomorphism $X:\BOm\to\BOm$ and a positive function $g\in C^\infty(\overline\Om)$ such that $\lB:= g X_*B_0$ is an MHS equilibrium on~$\Om$, that is,
	\begin{equation}\label{E.statEuler2}
	\lB\times \curl \lB+\nabla P=0\,,\qquad \Div \lB=0\,,\qquad \lB|_{\pd\Om}\cdot N=0		
	\end{equation}
	for some $P\in C^\infty(\BOm)$. Furthermore, the set of vector fields satisfying conditions (i)--(iii) is dense in a nonempty open subset of~$\tcB$.
\end{theorem}

\bibliographystyle{amsplain}

\end{document}